\documentclass[11 pt]{amsart}
\usepackage{comment}
\usepackage{enumerate}
\usepackage{amssymb, amsmath}
\usepackage{bm}
\title[Thue inequalities with few coefficients]{Thue inequalities with few coefficients}

 \author{Paloma Bengoechea}
\address{ ETH, Mathematics Dept.\\CH-8092, Z\"urich, Switzerland}
\email{paloma.bengoechea@math.ethz.ch}


\def\x{\langle \bm x\rangle}
\def\xx{|\bm x|}

\def\GL{\mathrm{GL}(2,\mathbb{Z})}
\def\SL{\mathrm{SL}(2,\mathbb{Z})}

\def\Z{\mathbb{Z}}

\begin{document}

\newtheorem{thm}{Theorem}[section]
\newtheorem{prop}[thm]{Proposition}
\newtheorem{lemma}[thm]{Lemma}
\newtheorem{cor}[thm]{Corollary}
\newtheorem{conj}[thm]{Conjecture}
\newtheorem{rk}[thm]{Remark}
\newtheorem{defi}[thm]{Definition}

\begin{abstract}
Let $F(x, y)$ be a binary form  with integer coefficients, degree $n\geq 3$, and irreducible over the rationals.
Suppose that only $s + 1$ of the
$n + 1$ coefficients of $F$ are nonzero. We show that the Thue inequality $|F(x,y)|\leq m$ has 
$\ll s m^{2/n}$ solutions provided that  the absolute value of the discriminant $D(F)$ of $F$ is large enough. 
We also give a new upper bound for the number of solutions of  $|F(x,y)|\leq m$, with no restriction on the discriminant of $F$
that depends mainly on $s$ and $m$, and slightly on $n$. 
Our bound becomes independent of $m$ when  $m<|D(F)|^{2/(5(n-1))}$, and also independent of $n$  
if $|D(F)|$ is large enough. 
%Let $F(x, y)$ be a binary form of degree $n \geq 3$ with integer coefficients, and irreducible over the rationals.
%We obtain a new upper bound for the number of solutions of the Thue inequality $|F(x,y)|\leq m$, when $F$ is a sparse form, in terms
%of the number of non-vanishing coefficients and $m$. Our 
%bound becomes independent of $m$ when $m$ is small in terms of the discriminant $D(F)$ of $F$ (namely when $m<D(F)^{1/3(n-1)}$). We use this bound 
%to establish, when $D(F)$ is large enough, a sharp upper bound for the number of solutions
%that is linear in terms of the number of non-vanishing coefficients. 
\end{abstract}
\maketitle

\section{Introduction}

%Bombieri and Schmidt prove linear in $n$ when $m=1$ and a constant $215$ when $n$ is large. 
%Then they deduce the bound $\ll n^{1+w(m)}$ for general $m$. Shabnam gives the explicit constant 
%for $m=1$ under the condition that $D>D_0(n)$ and this constant is sharp.

Let $F(x , y)$ be a binary form  with integer  coefficients and degree $n\geq 3$, irreducible over the rationals.  
Let $m$ be a positive integer.  Thue studied in \cite{Thu} the inequalities %In this paper we give upper bounds to the Thue inequalities
\begin{equation}\label{1}
1\leq |F(x , y)| \leq m,
\end{equation}
known as \emph{Thue inequalities}, showing that they have  finitely many solutions in integers $x$ and $y$. 
%By a classical result of Thue 
%in  \cite{Thu}, we know that the inequality.
%has at most finitely many solutions in integers $x$ and $y$. Such inequalities are called \emph{Thue inequalities}. 
%Thue showed in that the inequalities \eqref{1} have finitely many solutions in integers $x$ and $y$ and Mahler \cite{Mah9} bounded the number of 
%solutions by $c(F)m^{2/n}$, where $c(F)$
Mahler \cite{Mah9} showed that Thue inequalities have at most $c(F)m^{2/n}$ solutions, where $c(F)$
depends only on $F$. In this bound the dependence on $m$ is best possible if $m$ is large. 
For Thue equations $|F(x,y)|=m$, the dependence on $F$ of $c(F)$ has been progressively replaced by a dependence on the degree $n$, first by Siegel
in some special cases, and in general by Evertse \cite{Evt} 
  in his thesis. Later Bombieri and Schmidt \cite{Bom} obtained the bound $\ll n^{1+\nu}$ 
for the number of primitive solutions (solutions $(x,y)$ with $x$ and $y$ coprime), where $\nu$ is the number of prime factors of $m$.

In his fundamental work on diophantine equations $f(x, y)=0$, Siegel \cite{Sie} conjectured that,
 when the curve defined by the equation is irreducible and of positive genus,
 the number of solutions  sould be bounded only in terms of the number of nonzero coefficients. In this form, the conjecture is not true; 
 there is no bound independent of $m$ for cubic Thue equations, as the work of 
 Chowla \cite{Ch}, 
 Mahler \cite{Mah35} and Silverman \cite{Sil}  show.
 However, there have been several subsequent works
 with the goal of  replacing the dependence on the degree by the number of nonzero coefficients. Schmidt was the first in studying this modified version of
 Siegel's conjecture for Thue equations in general, and it turned out to be equally difficult to study  Thue inequalities (see his introduction
  in \cite{S87}).
  
 Suppose  that $F(x , y)$  has not more than $s+1$ nonzero coefficients, so that
\begin{equation}\label{F}
F(x,y)=\sum_{i=0}^s a_i x^{n_i}y^{n-n_i}
\end{equation}
with $0=n_0<n_1<\ldots n_{s-1}<n_s=n$. Then Schmidt \cite{S87} proved that the inequality \eqref{1} has
\begin{equation}\label{Sbound}
\ll (ns)^{1/2}m^{2/n}(1+\log m^{1/n})
\end{equation}
solutions. Here and throughout the paper,  the constants implicit
in $\ll$ will be absolute and effectively computable. Thunder \cite{Thu} could remove the logarithmic factor for many values of $m$.
Later, Mueller and Schmidt \cite{MS88} obtained the second bound
\begin{equation}\label{MSbound}
\ll s^2m^{2/n}(1+\log m^{1/n}),
\end{equation}
hence  the number of solutions of \eqref{1} is bounded  in terms of $s$ and $m$ only. This was proved previously for $s=1$ 
(i.e. for binomial forms) in \cite{Mu} and for $s=2$ (i.e. for trinomials) in \cite{MS87}.
Mueller and Schmidt could remove the logarithmic factor in their general bound \eqref{MSbound} if $n\geq \max(4s,s\log^3s)$. 
When $n\gg s$, $F$ is usually called \emph{sparse form} or \emph{fewnomial}. 
They also conjectured that the logarithmic factor should be removed for all forms of degree $n\geq 3$ and, more importantly, that the term $s^2$ 
should  be $s$.  

Here we establish two new upper bounds for the number of solutions of \eqref{1}. The bound given in Theorem \ref{th2} proves Mueller-Schmidt's conjecture 
for almost all binary forms with given degree. %Hence it
%improves significantly the bounds
%\eqref{Sbound} and \eqref{MSbound} under a (quite) strong condition on the discriminant.
\begin{thm}\label{th2}
  Let $F(x,y)\in\mathbb{Z}[x,y]$ be an irreducible  binary form with  $s+1$ nonzero coefficients and degree $n\geq 3$.
Assume that the absolute value of the discriminant
of $F$ is greater than $(n(n-1))^{8n(n-1)}$.
For each positive integer $m$, the inequality $|F(x , y)| \leq m$ has
$$\ll  s  m^{2/n}$$
 solutions. 
\end{thm}

%This sort of assumption on the discriminant appears in previous works, for example in \cite{Aun}, \cite{AB}. 
%This was also the case in 
%\cite{Aun}, where the bound $11n-2$  is given for the Thue equation $|F(x,y)|=1$, improving the implicit constant that one would obtain in 
%Bombieri-Schmidt's work \cite{Bom}.
% as we can see in Theorem \ref{th2}, and in other works such as \cite{Aun}.
Since  
  there are only finitely many $\SL$-equivalence classes of irreducible binary forms of fixed degree and bounded discriminant (see \cite{BiMer}),
 our result, while stated for a quite strong condition on the discriminant, 
holds for almost all classes of forms  of  given degree.  (Note that equivalent forms give the same number of solutions to the inequality \eqref{1}.)

Under a similar condition on the discriminant, an upper bound for the number of solutions to \eqref{1} for 
small values of $m$ and almost all forms is given in \cite{AB}, following previous works by \cite{EG16} and \cite{Gyo1}. 
That bound is linear in $s$ when the forms are `very' sparse, namely when $n\geq s^2$.

In Theorem \ref{th1} we give a new bound that holds for all sparse forms and all integers $m$. It becomes
independent of $m$ for small values of $m$ as Corollary \ref{cor} shows, and also independent of $n$ 
when the absolute value of the discriminant is large enough.

\begin{thm}\label{th1}
  Let $F(x,y)\in\mathbb{Z}[x,y]$ be an irreducible  binary form with  $s+1$ nonzero coefficients and degree $n\geq 3s$. Let $D(F)$ and $H(F)$ be the 
  discriminant and the height of $F$ respectively.
For each positive integer $m$  %be a positive integer %with 
%\begin{equation}\label{m} 
%m < M^{\frac{1}{8}(1-\frac{s}{n-s})}.
%\end{equation}
the inequality 
\begin{equation}\label{Tin2}
|F(x , y)| \leq m
\end{equation} 
 has 
 $$\ll (c(s)(1+\log m^{\frac{1}{n}})+\log^3 n) m^{\frac{2}{n}}|D(F)|^{-\frac{1}{n(n-1)}}
 $$
solutions, with
\begin{equation}\label{c}
c(s)=\left\{\begin{array}{ll}
s &\quad\mbox{if $n\geq s^4$}\\ 
s\log s &\quad\mbox{if $9s^2 \leq n<s^4$}\\
 s\log s(1+\frac{s}{\log H(F)}) &\quad\mbox{if $n<9s^2$}.
\end{array}\right.
\end{equation}
\end{thm}

\begin{cor}\label{cor}
Let $F(x,y)\in\mathbb{Z}[x,y]$ be an irreducible  binary form with  $s+1$ nonzero coefficients and degree $n\geq 3s$. Let $D(F)$ 
be the discriminant of $F$ and $m$ be a positive integer such that
 \begin{equation}\label{ccm}
 m\leq|D(F)|^{\frac{1}{(2+\frac{1}{2})(n-1)}}.
 \end{equation}
 Then $|F(x,y)|\leq m$ has 
$$  \ll c(s)+ \log^3 n 
$$
 solutions, with $c(s)$  defined by \eqref{c}.
 
Moreover, if $|D(F)| \geq (\log n)^{15n(n-1)}$,
 then $|F(x,y)|\leq m$ has 
 $$\ll c(s)$$ 
 solutions.
\end{cor}

The corollary immediately follows from Theorem \ref{th1} on noticing that $\log |D(F)|^{\frac{1}{(2+1/2)n(n-1)}}\leq|D(F)|^{\frac{1}{2(2+1/2)n(n-1)}}$, 
so when $m$ satisfies \eqref{ccm}, 
we have that
\begin{align*}
\log m^{1/n}  m^{2/n} &\leq %\log |D(F)|^{\frac{1}{(2+\frac{1}{2})n(n-1)}} |D(F)|^{\frac{2}{(2+\frac{1}{2})n(n-1)}}
|D(F)|^{\frac{1/4}{(1+\frac{1}{4})n(n-1)}} |D(F)|^{\frac{1}{(1+\frac{1}{4})n(n-1)}}\\
%\leq |D(F)|^{\frac{\frac{1}{4}}{(1+\frac{1}{4})n(n-1)}} |D(F)|^{\frac{1}{(1+\frac{1}{4})n(n-1)}}
&\leq |D(F)|^{\frac{1}{n(n-1)}},
\end{align*} and when $|D(F)| \geq (\log n)^{15n(n-1)}$,
$$
m^{2/n} \log^3 n \leq |D(F)|^{\frac{1}{n(n-1)}}.
$$

%We remark here that  we may take $D_0 =  n^{n(n-1)\log^2n}$. 

Note that we will regard $(x,y)$ and $(-x,-y)$ as one solution, and can assume $y \geq 0$ or $x\geq 0$ if convenient.

\section{Preliminaries}\label{secNot}\label{pre}
%----------------------------------------

%---------------
\subsection{ Discriminant, Height, and Mahler Measure}
%--------------------

For a binary form $G(x , y)$ that factors over $\mathbb{C}$ as
$$
\prod_{i=1}^{n} (\alpha_{i} x - \beta_{i}y),
$$
the discriminant $D(G)$ of $G$ is given by
$$
D(G) = \prod_{i<j} (\alpha_{i} \beta_{j} - \alpha_{j}\beta_{i})^2.
$$
Therefore, if we write
$$
G(x, y) = a_n (x - \gamma_{1}y)\ldots  (x- \gamma_{n}y),
$$ 
we have
$$
D(G) = a_n ^{2(n-1)} \prod_{i<j} (\gamma_{i}  - \gamma_{j})^2.
$$

The Mahler measure $M(G)$ is defined by
$$
M(G) = |a_n| \prod_{i =1}^{n} \max(1 , \left| \gamma_{i}\right|).
$$
 Mahler \cite{Mah} showed
\begin{equation}\label{mahD5}
 M(G) \geq \left(\frac{|D(G)|}{n^n}\right)^{\frac{1}{2n -2}}.
 \end{equation}
 
 If we write  $G(x , y)=a_{n}x^{n} + a_{n-1}x^{n-1} y+ \ldots + a_{1}x y^{n-1} + a_{0}y^n$,
 the (naive) height of $G$, denoted by $H(G)$, is defined by 
\begin{equation}\label{height}
H(G) =\max  \left( |a_{n}|, |a_{n-1}|, \ldots , |a_{0}|\right).
\end{equation}
We have 
\begin{equation}\label{Lan5}
  {n \choose \lfloor n/2\rfloor}^{-1} H(G) \leq M(G) \leq (n+1)^{1/2} H(G).
\end{equation}
A proof of this fact can be found in \cite{Mahext}.

%---------------------
\subsection{ $GL_{2}(\mathbb{Z})$ Actions and Equivalent Forms}
%----------------------------

Let $A = \left( \begin{array}{cc}
a & b \\
c & d \end{array} \right)$
and define the binary form $F_{A}$ by
$$
F_{A}(x , y) = F(ax + by , cx + dy).$$
Note that
\begin{equation}\label{St6}
D(F_{A}) = (\textrm{det} A)^{n (n-1)} D(F),
\end{equation}
and $D(F_A)=D(F)$ if $A\in\GL$.

We say that two binary forms $F$ and $G$ are equivalent if $G = \pm F_{A}$ for some $A \in GL_{2}(\mathbb{Z})$.
The number of solutions (and the number of primitive solutions) to Thue  inequalities does not change if we replace the binary form with an 
equivalent form.  However,  $GL_2(\mathbb{Z})$-actions do not preserve  the number of 
nonzero coefficients of $F$.
 Schmidt formulates in \cite{S87}  a condition that is invariant under
$GL_2(\mathbb{Z})$ actions.
He defines a class $C(t)$ of forms of fixed degree as follows. \\
\textbf{Definition of $C(t)$}. We define the set $C(t)$ as the set of forms
$F(x, y)$ of degree $n$ with integer coefficients, and irreducible over $\mathbb{Q}$, such that for any
reals $(u,v)\neq (0,0)$,
the form 
\begin{equation}\label{C(t)}
uF_x+vF_y 
\end{equation}
has at most $t$ real zeros. 

Note that for $n > 0$, the irreducibility of $F$ implies that
the form \eqref{C(t)} of degree $n - 1$ is not identically zero. Note also that for $F\in C(t)$,
the derivative $F_x(z, 1)$ has $\leq t$ real zeros and $F(z,1)$ has  $\leq t+1$ real zeros. The following is Lemma 2 of \cite{S87}.
 \begin{lemma} \label{Ct}
 Suppose $F(x ,y)$ is irreducible of degree $n$, and  has  $s+1$ non-vanishing coefficients. Then $F(x , y) \in C(4s -2)$. 
  \end{lemma}

%In Theorem \ref{th1}, we will consider inequalities of the shape $|F(x , y)| \leq h$, for forms $F \in C(4s -2)$.

% in  \cite[eq. 2.10, 2.9]{MS88} to distinguish between 
%small and large solutions. Therefore we can use their result \cite[Prop. 1, p.211]{MS88} for the count of large solutions:

\section{General strategies for Theorems \ref{th2} and \ref{th1}}

\textbf{Definition of Primitive Solutions}. A pair $(x , y) \in \mathbb{Z}^2$ is called a primitive solution to the inequality \eqref{1}  if it satisfies the inequality and  
$\gcd(x , y) = 1$.  \\
We note that by this definition  the possible solutions $(z, 0)$ and $(0, z)$ are considered primitive if and only  if $z = \pm 1$. 
\\

\textbf{Definitions of $N(F,m)$, $P(F,m)$ and $\tilde P(F,m)$.}
For an irreducible binary form $F(x,y)\in\Z[x,y]$ of degree $n\geq 3$, we denote by 
$N(F,m)$ the number of solutions of $F(x,y)\leq m$ and by $P(F,m)$ the number of primitive solutions. Further, we write 
 $\tilde P(F,m)$ for the number of primitive solutions  of
\begin{equation}\label{P}
2^{-n}m\leq F(x,y) <m.
\end{equation}
Note that $\tilde P(F,m)$ is not affected if we replace $F$ by an equivalent form. 
We will show that for $F$ of the form \eqref{F},
\begin{equation}\label{b2}
\tilde P(F,m) \ll s m^{2/n} \qquad\mbox{for $|D(F)| >(n(n-1))^{8n(n-1)}$}
\end{equation}
and
\begin{equation}\label{b1}
\tilde P(F,m) \ll (c(s)(1+\log m^{\frac{1}{n}})+\log^3 n) m^{\frac{2}{n}}|D(F)|^{-\frac{1}{n(n-1)}}
\end{equation}
with no restriction on the discriminant $D(F)$.

Once we obtain these upper bounds for $\tilde P(F,m)$, it is easy to deduce the same upper bounds for
$N(F,m)$. We follow the argument in  \cite[section 3]{S87}. Write 
$$
\tilde P(F,m)=(A_1(F)+A_2(F)(1+\log m^{\frac{1}{n}}))m^{\frac{2}{n}},
$$
where $A_1(F), A_2(F)$ depend only on $F$.

When $u$ is the integer with $2^{nu}\leq m< 2^{n(u+1)}$, then
\begin{align}
P(F,m)\leq P(F,2^{n(u+1)}-1) &= \sum_{j=1}^{u+1} \tilde P(F,2^{nj}) \nonumber\\
&\ll A_1(F) \sum_{j=1}^{u+1} 2^{2j} + A_2(F)\sum_{j=1}^{u+1} 2^{2j}(1+\log 2^j)\nonumber\\
&\ll A_1(F) 2^{2u} + A_2(F)2^{2u}(1+u)\nonumber\\
&\ll (A_1(F) + A_2(F)(1+\log m^{\frac{1}{n}}))m^{\frac{2}{n}}. \label{PFm}
\end{align}

Assume \eqref{b1} and \eqref{b2}. Then \eqref{PFm} implies
\begin{equation}\label{b11}
 P(F,m)\ll (c(s)(1+\log m^{\frac{1}{n}})+\log^3 n) m^{\frac{2}{n}}|D(F)|^{-\frac{1}{n(n-1)}}
\end{equation}
and
\begin{equation}\label{b22}
 P(F,m)\ll s m^{2/n} \qquad\mbox{for $|D(F)| >(n(n-1))^{8n(n-1)}$}. 
\end{equation}

Let $\pi(F,m)$ be the number of primitive solutions of $F(x,y)=m$.
Then 
$$
\pi(F,m)=P(F,m)-P(F,m-1)
$$
(with $P(F,0)=0$). With $[\cdot]$ denoting integer part, we have
\begin{align*}
N(F,m)&=\sum_{k=1}^m \pi(F,k) \left[\left(\dfrac{m}{k}\right)^{1/n}\right] \leq m^{1/n} \sum_{k=1}^m \pi(F,k) k^{-1/n}\\
&=m^{1/n} \sum_{k=1}^m (P(F,k)-P(F,k-1)) k^{-1/n}\\
&\ll P(F,m) m^{-1/n} \sum_{k=1}^{m} (k^{2/n} - (k-1)^{2/n})k^{-1/n}\\
&\ll P(F,m) m^{-1/n} \sum_{k=1}^{m}  k^{1/n} - (k-1)^{1/n}\\
&\ll P(F,m)
\end{align*}
since the sum is telescoping.

Hence the whole difficulty in Theorems \ref{th2} and \ref{th1}  is to bound  $\tilde P(F,m)$, the number of solutions to \eqref{P}. 
We will split the count of  possible solutions to \eqref{P} into small and large solutions for Theorem \ref{th2}, and
small, medium and large solutions for Theorem \ref{th1}. The definitions of small and large
will differ for the two theorems. However, a common argument is used for small solutions.  
We use the classical decomposition of $F(x,y)$ into linear forms introduced by Bombieri and Schmidt in \cite{Bom} and used in several works afterwards
to estimate small solutions.
We will also use a lemma by Mueller and Schmidt (recorded here as Lemma \ref{lmMS88}) in a similar way as in 
\cite[section 4]{AB}. 
This lemma  is crucial for the treatment of large solutions in Theorem \ref{th2}. Mueller and Schmidt formulated their lemma in terms
of forms with $s$ nonzero coefficients, but in fact this lemma can be applied to a larger class of forms, and this is the reason why we exploit it so much
in this paper. We combine it with an argument from \cite{Thu}, which is based on  
the Lewis-Mahler inequality (Lemma \ref{Ste}) on the approximation  by the roots of $F$  to the rationals $\frac{x}{y}$, where $(x,y)$ are integral 
solutions to \eqref{1}, together with a gap type result due to Schmidt \cite{Sbook}. 
%Mueller and Schmidt's Lemma \ref{}.

The bound for medium solutions for Theorem \ref{th1} is  an extension of the argument in \cite[section 5]{AB}. The bound 
for large solutions  is given by a result of Mueller and Schmidt in \cite{MS88}.
In the calculation of the three bounds (for small, medium and large solutions) for Theorem \ref{th1}, we use some results that need the 
assumption that $F$ has $s$ nonzero coefficients, and some other results that need the assumption that $F$ has minimal Mahler measure. Combining both
 assumptions can be a problem a priori, and this may be a reason for the existence of the two simultaneous papers \cite{S87} and \cite{MS88}, where
 each of them assumes exactly one of the two hypothesis. The way we are able to combine both hypothesis here (in section \ref{sP1}) is also new.
\\

The rest of the manuscript is organized as follows. We discuss the argument for small solutions for both theorems in section \ref{ssmall}, and then 
apply it to the specific definitions of `small' in sections   \ref{ssmall2} and \ref{ssmall1}.
In sections 5-7 we focus on Theorem \ref{th2}. In section \ref{sP2} we give the definitions of small and large solutions,  we give the results that 
essentially count them (see Propositions \ref{p2small} and \ref{p2large}) and bound
$\tilde P(F,m)$ for large discriminants assuming them.
We prove Propositions \ref{p2small} and \ref{p2large} in sections \ref{ssmall2} and \ref{slarge} respectively. 

 In sections 8-10 we focus on Theorem \ref{th1}. In section \ref{sP1} we give the definitions of small, 
medium and large solutions, and again we give the results  that essentially count those solutions 
(see Propositions \ref{ppsmall}, \ref{ppmedium} and \ref{pplarge}), 
and bound $\tilde P(F,m)$ assuming them.
In sections  \ref{ssmall1} and \ref{smed} we prove Propositions \ref{ppsmall} and \ref{ppmedium} respectively.

%----------------------------------------------------------------------------------------------------------------------------------

  \section{Small solutions}\label{ssmall}

  %------------------------------------------------------------------------------

  Let $F(x,y)\in\mathbb{Z}[x,y]$ be an irreducible binary form of degree $n\geq 3$ that lies in $C(4s-2)$.
  Let $M$ be the smallest Mahler measure
among the forms equivalent to $F$, and  $m$ be a positive integer such that
\begin{equation}\label{m2}
 m\leq \dfrac{M}{100^n}.
 \end{equation}

 Under these assumptions,
   we give an upper bound for the number of solutions $(x,y)$ of \eqref{P} that satisfy $1\leq y <Y$, for a constant $Y$. Our bound will 
of course depend on $Y$ and will be applied later to two different values of $Y$. Similarly we give an upper bound for the number of solutions with 
$1\leq x <Y$.

We suppose that there is at least one primitive solution of \eqref{P} with $1\leq y <Y$. We fix  such a solution
$(x^\ast,y^\ast)$  such that 
$y^\ast\leq  y$ for all primitive  solutions $(x,y)$. Note that any primitive solution $(x,y)\neq(\pm 1,0)$ will have $y\geq 1$.
\\

 %We will regard $(x,y)$ and $(-x,-y)$ as one solution, and can assume $y \geq 0$, if we need to.

\textbf{Definition of $L_i(x,y)$.} For the binary form
$$
F(x,y)=a_n(x-\alpha_1y)\ldots(x-\alpha_ny),
$$
we define
$$
L_i(x,y)=x-\alpha_iy\qquad (i=1,\dots,n).
$$
Here $\alpha_1,\ldots,\alpha_n$ are the roots of the polynomial $F(x,1)$.

Given $\bm{x} = (x , y)$ and $\bm{x'} = (x' , y')$, we define
\begin{equation}\label{defofD}
\mathcal{D}(\bm x,\bm x')=xy'-x'y.
\end{equation}

\begin{lemma}\label{S56}
Suppose $\bm{x}=(x,y)$ is a primitive solution of \eqref{P}. We have
$$
\frac{L_i(x^\ast,y^\ast)}{L_{i}(x , y)} - \frac{L_j(x^\ast,y^\ast)}{L_{j}(x , y)} = (\beta_{j} - \beta_{i}) \mathcal{D}(\bm x,\bm x^\ast),
$$
where $\beta_{1}$,\ldots, $\beta_{n}$ depend on $(x,y)$  and are such that the form
$$
J(u , w) = F(x,y) (u - \beta_{1}w)\ldots (u - \beta_{n}w)
$$
is equivalent to $F$.
\end{lemma}
\begin{proof}
 This is  Lemma 3 of \cite{Bom}.
\end{proof}

For a primitive solution $(x,y)$ of \eqref{P}, we have
\begin{equation}\label{F/F}
 \dfrac{F(x^\ast,y^\ast)}{F(x,y)} < 2^n.
\end{equation}
Let $i_0\in\{1,\ldots,n\}$ be the index such that  
\begin{equation}
 \left|\dfrac{L_{i_0}(x^\ast,y^\ast)}{L_{i_0}(x,y)}\right|=\min_{1\leq i\leq n}\left|\dfrac{L_i(x^\ast,y^\ast)}{L_i(x,y)}\right|,
 \end{equation}
 so that, by \eqref{F/F},
 \begin{equation}\label{i1}
 \left|\dfrac{L_{i_0}(x^\ast,y^\ast)}{L_{i_0}(x,y)}\right|\leq 2.
 \end{equation}
 
By Lemma \ref{S56} and \eqref{i1},
\begin{equation}\label{S57}
\frac{|L_i(x^\ast,y^\ast)|}{\left|L_{i}(x , y) \right|} \geq |\beta_{i_0} - \beta_{i}| |\mathcal{D}(\bm x,\bm x^\ast)| - 2.
\end{equation}
For the complex conjugate  $\bar{\beta}_{i_0}$  of $\beta_{i_0}$, we also have
$$  
\frac{|L_i(x^\ast,y^\ast)|}{\left|L_{i}(x , y) \right|} \geq |\bar{\beta}_{i_0} - \beta_{i}| |\mathcal{D}(\bm x,\bm x^\ast)| - 2.
$$
Hence
$$
\frac{|L_i(x^\ast,y^\ast)|}{\left|L_{i}(x , y) \right|} \geq |\textrm{Re}(\beta_{i_0}) - \beta_{i}| |\mathcal{D}(\bm x,\bm x^\ast)| - 2,
$$
where $\textrm{Re}(\beta_{i_0})$ is the real part of $\beta_{i_0}$.
Now we choose an integer $d = d(x , y)$, with $|\textrm{Re}(\beta_{i_0}) -d|\leq 1/2$, and we obtain 
\begin{equation}\label{S58}
\frac{|L_i(x^\ast,y^\ast)|}{\left|L_{i}(x , y) \right|} \geq \left(|d- \beta_{i}| -\frac{1}{2}\right) |\mathcal{D}(\bm x,\bm x^\ast)| -2,
\end{equation}
for $i = 1,\ldots , n$.
\\

\noindent \textbf{Definition of the sets  $\frak{X}_{i}$}.
Let $\frak{X}_{i}$ be the set of primitive solutions $(x,y)\neq (x^\ast,y^\ast)$ of \eqref{P} with $1\leq y\leq Y$ and $|L_i(x,y)|\leq \frac{1}{2y}$, where $1\leq i \leq n$.

We note that if $\alpha_{i}$ and $\alpha_{j}$ are complex conjugates then $\frak{X}_{i} = \frak{X}_{j}$.

\begin{lemma}\label{Sl5}
 Suppose $(x , y)$ and $( x' , y')$ are two distinct primitive solutions in $\frak{X}_{i}$,  with $1\leq  y \leq  y'$ . Then
$$
\frac{ y'}{y} \geq \frac{2}{11} \max(1 , |\beta_{i}(x , y) - d(x , y)|).
$$
\end{lemma}
\begin{proof} We follow the proof of Lemma 4 of \cite{Bom} and Lemma 4.3 of \cite{AB}. 
We have that 
\begin{align}\label{DxSl5}
1\leq 
\left| y' x - y   x'\right|
&\leq y |L_i( x', y')|+ y' \left| L_i(x ,y)\right|\\
&\leq\dfrac{y}{2 y'}+  y' \left|L_i(x,y)\right| \nonumber \\
&\leq \dfrac{1}{2}+ y' \left|L_i(x,y)\right|. \nonumber
\end{align}
Therefore,
 $$ 
 y'\geq \frac{1}{2 \left| L_i(x,y)\right|}.
 $$
  Combining this with \eqref{S58}, we get
\begin{align}\label{pre32}
\frac{ y'}{y} &\geq \dfrac{1}{2}\Big(|d(x,y)-\beta_i(x,y)|-\dfrac{1}{2}\Big)\dfrac{|\mathcal{D}(\bm x,\bm x^\ast)|}{y|L_i(x^\ast,y^\ast)|} 
- \dfrac{1}{y|L_i(x^\ast,y^\ast)|} \nonumber \\
&\geq \dfrac{(|d-\beta_i|-\frac{5}{2}) |\mathcal{D}(\bm x,\bm x^\ast)|}{2y|L_i(x^\ast,y^\ast)|}.
\end{align}

Now, 
\begin{eqnarray}\label{32}\nonumber
\dfrac{|\mathcal{D}(\bm x,\bm x^\ast)|}{y |L_i(x^\ast,y^\ast)|} 
&=&\dfrac{\left|\dfrac{x}{y}-\dfrac{x^\ast}{y^\ast}\right|}{\left|\dfrac{x^\ast}{y^\ast}-\alpha_i\right|}
\geq
\dfrac{\left|\dfrac{x}{y}-\dfrac{x^\ast}{y^\ast}\right|}{\left|\dfrac{x^\ast}{y^\ast}-\dfrac{x}{y}\right|+\dfrac{1}{2y^2}}\\
&\geq&
\dfrac{1}{1+\dfrac{1}{2y^2\left|\dfrac{x^\ast}{y^\ast}-\dfrac{x}{y}\right|}}\geq \dfrac{2}{3},
\end{eqnarray}
where the last inequality is because  $\left|\dfrac{x^\ast}{y^\ast}-\dfrac{x}{y}\right|\geq \dfrac{1}{y^\ast y}$ and so
$$
\dfrac{1}{2y^2\left|\dfrac{x^\ast}{y^\ast}-\dfrac{x}{y}\right|}\leq \dfrac{y^\ast}{2y}\leq \dfrac{1}{2}.
$$
Therefore, by \eqref{pre32} and \eqref{32}, we have
\begin{align*}
\frac{ y'}{y} &\geq \max\Big(1,\Big(|d-\beta_i|-\dfrac{5}{2}\Big)\dfrac{1}{3}\Big)\\
&\geq \dfrac{3}{11} \max\Big(1,\dfrac{2}{3}|d-\beta_i|\Big),
\end{align*}
where in the second inequality  we used that $\max(1,\frac{\zeta}{2}-a)\geq \frac{1}{2a+2} \max(1,\zeta)$ with $\zeta=\frac{2}{3}|d-\beta_i|$ and 
$a=\frac{5}{6}$.
\end{proof}

\begin{lemma}\label{Sl6}
Suppose $(x , y)\not\in\frak{X}_{i}\cup\{(x^\ast,y^\ast)\}$ is a  primitive solution of \eqref{P} with $y>0$. Then 
$$
|d(x , y) - \beta_{i}(x , y)| \leq \dfrac{11}{2}.
$$
\end{lemma}
\begin{proof}
 By \eqref{S58},  we have 
$$
|d-\beta_i|\leq \Big(\dfrac{|L_i(x^\ast,y^\ast)|}{|L_i(x,y)|} + 2\Big)\dfrac{1}{|\mathcal{D}(\bm x,\bm x^\ast)|}+\dfrac{1}{2}.
$$
Since $\left|\dfrac{x}{y}-\dfrac{x^\ast}{y^\ast}\right|\geq\dfrac{1}{yy^\ast}$ and we are assuming $|L_i(x,y)|>\frac{1}{2y}$, we have
$$
\dfrac{|L_i(x^\ast,y^\ast)|}{|L_i(x,y)| |\mathcal{D}(\bm x,\bm x^\ast)|}
\leq \dfrac{\left|\alpha_i-\dfrac{x}{y}\right| + \left|\dfrac{x}{y}-\dfrac{x^\ast}{y^\ast}\right|}{y^2 \left|\alpha_i-\dfrac{x}{y}\right| \left|\dfrac{x}{y}-\dfrac{x^\ast}{y^\ast}\right|}
\leq \dfrac{1}{y^2 \left|\dfrac{x}{y}-\dfrac{x^\ast}{y^\ast}\right|} + \dfrac{1}{y^2 \left|\alpha_i-\dfrac{x}{y}\right|}
\leq 3.
$$
Therefore, using that   $|\mathcal{D}(\bm x,\bm x^\ast)|\geq 1$, we conclude that
$$
|d-\beta_i|\leq \dfrac{11}{2}.
$$
\end{proof}

Let  $(x,y)$ be a fixed primitive solution to \eqref{P}. Recall that the form 
$$
J(u , w) = F(x , y) (u - \beta_{1}w)\ldots (u - \beta_{n}w)
$$
is equivalent to $F$ (see Lemma \ref{S56}). Hence the form 
$$
\hat{J}(u , w) =  F(x , y) (u - (\beta_{1}-d) w)\ldots (u - (\beta_{n}- d)w)
$$
is also equivalent to $F$. Therefore, 
\begin{equation}\label{Spre60}
\prod_{i=1}^{n} \max(1, |\beta_{i}(x, y)-d(x , y)|) \geq \frac{M(\hat{J})}{F(x,y)} \geq \dfrac{M}{m}.
\end{equation}

\textbf{Definition of $\frak{X}$}. For each set $\frak{X}_{i}$ ($i = 1, \ldots, n$)  that is not empty,  let  
$(x^{(i)} , y^{(i)}) \in \frak{X}_{i}$ be the element with the largest value of $y$. Consider
 the set of primitive solutions of \eqref{P} that are not $(x^\ast,y^\ast)$ and with $1 \leq y \leq Y$ minus the elements 
$(x^{(1)} , y^{(1)})$, \ldots, $(x^{(n)} , y^{(n)})$. We define  $\frak{X}$ to be that set together with the elements 
$(x^{(i)},y^{(i)})\in\frak{X}_i\cap\frak{X}_j$
 ($i,j \in\{1, \ldots, n\}$) such that $(x^{(i)},y^{(i)})\neq (x^{(j)},y^{(j)})$.

The rest of this section is devoted  to bound the cardinal of 
$$\frak{X}\cup\left\{(x^{(i)},y^{(i)})\right\}_{1\leq i\leq n}.
$$

\begin{lemma}
For any fixed $i \in \{1,\ldots, n\}$, we have
\begin{eqnarray}\label{E3}
 \prod_{(x , y) \in  \mathfrak{X}} \dfrac{2}{11} \max\left(1, \left|\beta_{i}(x, y) - d(x, y) \right|\right)
 \leq Y.
\end{eqnarray}
\end{lemma}
\begin{proof}
Fix $i\in\{1,\ldots,n\}$.
Suppose that
 the set   $\frak{X}_{i}$ is not empty. We index the elements of $\frak{X}_{i}$ as 
$$
(x_{1}^{(i)}, y_{1}^{(i)}), \ldots, (x_{v}^{(i)}, y_{v}^{(i)}),
$$
 so that $y_{1}^{(i)} \leq \ldots \leq y_{v}^{(i)}$ (note that $(x_{v}^{(i)}, y_{v}^{(i)}) = (x^{(i)} , y^{(i)})$). By Lemma \ref{Sl5},
\begin{equation}\label{yk}
\dfrac{2}{11} \max\left(1, \left|\beta_{i}(x_{k}^{(i)}, y_{k}^{(i)}) - d(x_{k}^{(i)}, y_{k}^{(i)})\right|\right) \leq 
\frac{y_{k+1}^{(i)}}{y_{k}^{(i)}}
\end{equation}
for $k = 1, \ldots, v-1$, so we have
\begin{equation}\label{prod1}
 \prod_{(x , y) \in  \frak{X} \bigcap \frak{X_{i}} } \dfrac{2}{11}  \max\left(1, \left|\beta_{i}(x, y) - d(x, y) \right|\right) \leq Y.
\end{equation} 

For any solution $(x , y)\in \frak{X}$  that does not belong to  $\frak{X_{i}}$,  by Lemma \ref{Sl6}, we have
\begin{equation}\label{E2}
\frac{2}{11} \max\left(1, \left|\beta_{i}(x, y) - d(x, y) \right|\right)  \leq 1.
\end{equation}
This, together with \eqref{prod1}, completes the proof of the lemma.
\end{proof}

Next we will  establish inequalities similar to \eqref{E3} for the solutions $(x^{(i)}, y^{(i)})$ which do not belong to $\frak{X}$. %We need the following  result from \cite{MS88}.
Lemma \ref{lmMS88} will be crucial in order to find a clever way of counting those solutions in terms of $s$.
Lemma \ref{lmMS88} is in fact Lemma 7 of \cite{MS88},  although in \cite{MS88} the lemma is stated for polynomials  with few nonzero coefficients, whereas we state it for
any polynomial that belongs to $C(4s-2)$. The proof only involves the fact that  the polynomials  and their derivatives  have few real zeros;
we reproduce it here for convenience of the reader. 
Put
\begin{equation}\label{R}
R=n^{800\log^2n}.
\end{equation}
We first need the following result.

\begin{lemma}\label{lemma 6}
Let $f(z)$ be a polynomial of degree $n$ with real coefficients. Suppose that
$f(x) f'(x)\neq 0$ for real $x\in I$, where $I$ is an interval $X_1 <x <X_2$, or a half line $x <X_2$, or $x > X_1$.
Suppose there are $u> 1$ roots $\gamma_j =x_j +i y_j$ $(j= 1,\ldots,u)$ with real parts $x_j\in I$. Then there
is a root $\gamma_\ell$ among these $u$ roots such that for every real  $\zeta$,
$$
|\zeta-\gamma_\ell| < R \min_{1\leq i\leq u}|\zeta-\gamma_i|.
$$
\end{lemma}

\begin{proof} 
 This is Lemma 6 of \cite{MS88}.
\end{proof}

\begin{lemma}\label{lmMS88}
There is a set $\bm{S}$ of roots $\alpha_i$ of  $F(x,1)$ with  $\left|\bm{S}\right|\leq 12s - 3$ such that 
for any real  $\zeta$,
$$
\min_{\alpha_\ell\in S}|\zeta-\alpha_\ell| \leq R \min_{1\leq i\leq n}|\zeta-\alpha_i|.
$$
\end{lemma}

\begin{proof} Let $u=4s-2$ and $f(x)=F(x,1)$. Since $F(x,y)\in C(u)$, $f(x)$ has $\leq u+1$ real zeros and 
its derivative $f'(x)$ has $\leq u$ real zeros, 
so that $f(x)f'(x)$  has $\leq 2u+1$ real zeros. Thus the real numbers $x$ with $f(x)f'(x)\neq 0$ fall into
$\leq 2u+2$ intervals (or half lines) $I$. Let $\bm S$ consist on the one hand of the real zeros of $f(x)$,
and on the other hand, for each interval $I$ as above for which there are roots of $f(x)$ with real
part in $I$, pick a $\gamma_\ell$ according to Lemma \ref{lemma 6}. The set $\bm S$ so attained will have
$|\bm S|\leq u+1 + 2u+2 = 3u+3=12s-3$.
\end{proof}

Let $\bm S=\{\alpha_1,\ldots,\alpha_{t}\}$, with $t\leq 12s-3$.
\\

\textbf{Definition of the set $\frak{X}_0$}.
Let 
 $$ 
 \frak{X}_0 = \left\{(x^{(i)},y^{(i)})\in\frak{X}_i\backslash\left\{\mathfrak{X}_1\cup\ldots\cup\mathfrak{X}_{t}\cup\mathfrak{X}\right\}\right\}_{t<i\leq n}.
 $$

Let $(x^{(i)},y^{(i)})\in \frak{X}_0$. By Lemma \ref{lmMS88} there exists $\ell\in\left\{1,\ldots,t\right\}$ such that
$$
\left|L_i(x^{(i)},y^{(i)})\right|\geq \dfrac{\left|L_\ell(x^{(i)},y^{(i)})\right|}{R} 
\geq \dfrac{1}{2y^{(i)}R},
$$
where the last inequality is because $(x^{(i)},y^{(i)}) \not \in \mathfrak{X}_{\ell}$.  Combining this with \eqref{S58}, we obtain
$$
\left|d(x^{(i)},y^{(i)})-\beta_{i}(x^{(i)},y^{(i)})\right| \leq\dfrac{2|L_i(x^\ast,y^\ast)|y^{(i)}R +2}{|\mathcal{D}(\bm x^{(i)},\bm x^\ast)|}+\dfrac{1}{2}.
$$
Using \eqref{32}, we obtain
\begin{equation}\label{extrapoints0}
\max\left(1,\left|d(x^{(i)},y^{(i)})-\beta_{i}(x^{(i)},y^{(i)})\right|\right) \leq 3R+\dfrac{5}{2}.
\end{equation}

Note that a solution $(x^{(j)},y^{(j)})$ belongs to $\frak{X}_i$ if and only if $(x^{(j)},y^{(j)})=(x^{(i)},y^{(i)})$ or $(x^{(j)},y^{(j)})\in\frak{X}$, and
in this case $(x^{(j)},y^{(j)})\not\in\mathfrak{X}_0$.
Hence,  by \eqref{extrapoints0} and Lemma \ref{Sl6},
\begin{equation}\label{extrapoints}
\prod_{(x,y)\in\mathfrak{X}_0} \dfrac{2}{11} \max\left(1,\left|d(x,y)-\beta_{i}(x,y)\right|\right) \leq \dfrac{6R+5}{11}
\end{equation}
for $i=1,\ldots,n$.
%Hence we have that
%$$
%\prod_{j=t}^n \prod_{i=1}^n \dfrac{\left|d(x^{(j)},y^{(j)})-\beta_{i}(x^{(j)},y^{(j)})\right|}{2} \leq \left(\dfrac{6R+1}{4}+1\right)^{n-t}.
%$$
Taking the product  of \eqref{E3} and \eqref{extrapoints} for $i=1,\ldots,n$  we find
\begin{align*}
\left(\dfrac{M}{(\frac{11}{2})^{n}m}\right)^{|\frak{X}\cup\mathfrak{X}_0|}\leq \left(Y\dfrac{6R+5}{11}\right)^n.
\end{align*}
Therefore,
\begin{equation}\label{takinglog}
|\mathfrak{X}\cup\mathfrak{X}_0| < \dfrac{n\log Y + n\log(6R+5)}{\log (M/(6^n m))}.
\end{equation}

The primitive solutions $(x,y)$ of \eqref{P} with $1\leq y\leq Y$ are in  
$$
\mathfrak{X}\cup\mathfrak{X}_0\cup\{(x^{(i)},y^{(i)})\}_{1\leq i\leq t}\cup\{(x^\ast,y^\ast)\},
$$
so there are $\ll |\mathfrak{X}\cup\mathfrak{X}_0| + s$ of them.

The number of primitive solutions of \eqref{P} with $1\leq x\leq Y$ can be estimated in a similar way, by considering the form 
$$
F(x,y)=a_0(y-\gamma_1 x)\cdot\ldots\cdot(y-\gamma_n x)
$$
and putting $L_i(x,y)=y-\gamma_i x$. Here $\gamma_1,\ldots,\gamma_n$ are the roots of the polynomial $F(1,y)$.

%--------------------------------------------------------------------------------------------

\section{Bound for $\tilde P(F,m)$ for large discriminants}\label{sP2}

%---------------------------------------------------------------------------------------------

Let $F(x,y)\in\mathbb{Z}[x,y]$ be an irreducible binary form of degree $n\geq 3$ that lies in $C(4s-2)$ such that
\begin{equation}\label{D0}
|D(F)|>(n(n-1))^{8n(n-1)}.
\end{equation}
Further, assume that $F$ has the smallest Mahler measure among all its equivalent forms. We will say for abreviation that $F$ has minimal Mahler measure.
Later we need this assumption in order to use simultaneously section \ref{ssmall} and Lemma \ref{Ste}. 
In Lemma \ref{Ste}, the Mahler measure of 
$F$ is involved, whereas in section \ref{ssmall} we work with the smallest Mahler measure among all forms equivalent to $F$. 
We need both measures to be the same.
Let $M=M(F)$ be the minimal Mahler measure.
We first count the number of solutions to \eqref{P} such that
\begin{equation}\label{m0}
 m\leq \dfrac{|D(F)|^{\frac{1}{2(n-1)}}}{e^{200n}}.
\end{equation} 
Note that, if $m$ satisfies \eqref{m0}, then by  \eqref{mahD5} we also have
\begin{equation}\label{m1}
m\leq \dfrac{M}{e^{200n}}.
\end{equation}

Let
\begin{equation}\label{YS2}
Y_0=(M/m)^5.
\end{equation}

 Relative to the quantity $Y_0$,  we call a 
solution $(x , y) \in \mathbb{Z}^2$ 
\begin{eqnarray*}
\textrm{small} \, \, &\textrm{if}& \, \, 0 \leq y\leq Y_0,  \\
\textrm{large} \, \, &\textrm{if}& \, \, y>Y_0.
\end{eqnarray*}

\begin{prop}\label{p2small}
Let $F\in C(4s-2)$ and $m$ be a positive integer satisfying \eqref{m0}. Assume \eqref{D0} and that $F$ has minimal Mahler measure.
The inequality \eqref{P} has $\ll n$ small solutions.
\end{prop}

\begin{prop}\label{p2large}
Let $F\in C(4s-2)$ and $m$ be a positive integer satisfying \eqref{m0}. Assume \eqref{D0} and that $F$ has minimal Mahler measure.
The inequality \eqref{P}  has $\ll s\log \log m^{1/n}$ large solutions.
\end{prop}

Hence the number of solutions of \eqref{P} with $F$ having minimal Mahler measure, with \eqref{D0}, and $m$ satisfying \eqref{m0} is
\begin{equation}\label{00}
\ll n + s\log\log m^{1/n}.
\end{equation}

Now we use the argument in \cite[section II]{Bom} to derive an upper bound for $\tilde P(F,m)$ for any positive integer $m$
and with no need that $F$ has minimal Mahler measure.
Pick the smallest prime $p$ that satisfies
\begin{equation}\label{p}
p \geq e^{400} m^{\frac{2}{n}} |D(F)|^{-\frac{1}{n(n-1)}}.
\end{equation}
Note that 
\begin{equation}\label{2p}
 p< 2e^{400} m^{\frac{2}{n}} |D(F)|^{-\frac{1}{n(n-1)}}.
\end{equation}

Let
$$
A_0=\begin{pmatrix} 1 &0\\0 &p\end{pmatrix},\qquad
A_j=\begin{pmatrix} p &j\\0 &1\end{pmatrix}\qquad (j=1,\ldots,p).
$$
We have that $F_{A_j}\in C(4s-2)$ for $j=0,\ldots,p$. 
We also have $\mathbb{Z}^2=\cup_{j=0}^p A_j\mathbb{Z}^2$, so that
\begin{equation}\label{Pineq1}
\tilde P(F,m)\leq \sum_{j=0}^p \tilde P(F_{A_j},m).
\end{equation}
Let $F'_{A_j}$ be a form equivalent to $F_{A_j}$ that has minimal Mahler measure. 

%For $j=0,\ldots,p$, by \eqref{St6} and \eqref{p}, 
%$$
%|D(F_{A_j})|=p^{n(n-1)}|D(F)|> 10^{6n(n-1)} m^{2n-2}.
%$$
%Then, by \eqref{mahD5},
%\begin{equation}\label{Dp}
%10^{6n(n-1)} m^{2n-2}\leq |D(F_{A_j})| \leq n^n M_j^{2n-2},
%\end{equation}
%where $M_j$ is the smallest Mahler measure among the forms equivalent to $F_{A_j}$.
For $j=0,\ldots,p$, by \eqref{St6} and \eqref{p}, 
$$
|D(F'_{A_j})|=p^{n(n-1)}|D(F)|\geq e^{400n(n-1)} m^{2n-2},
$$
so
$m\leq |D(F'_{A_j})|^{\frac{1}{2(n-1)}}/e^{200n}$, and
 $|D(F'_{A_j})|$ satisfies \eqref{D0},
so we can apply Propositions \ref{p2small} and \ref{p2large} to $F'_{A_j}$. Hence 
$$
\tilde P(F_{A_j},m)=\tilde P(F'_{A_j},m)\ll n + s\log\log m^{1/n}.
$$
By \eqref{Pineq1} and \eqref{2p}, we obtain that 
\begin{align}\label{tP2}
\tilde P(F,m) &\ll (p+1)(n + s\log\log m^{\frac{1}{n}})\nonumber\\
 &\ll (n + s\log\log m^{\frac{1}{n}}) m^{\frac{2}{n}} |D(F)|^{-\frac{1}{n(n-1)}}.
\end{align}

If $|D(F)|^{\frac{1}{n(n-1)}} > \log \log m^{\frac{1}{n}}$,  \eqref{b2} follows from \eqref{tP2} and the assumption \eqref{D0}.

If $|D(F)|^{\frac{1}{n(n-1)}} \leq \log \log m^{\frac{1}{n}}$, then by \eqref{D0} we also have that $\log\log m\geq (n(n-1))^8$, and the  
 result below, which is part of  the Corollary of Theorem 2 in \cite{Thu} (with $\varepsilon=1/2$) concludes the proof of \eqref{b2}.
\begin{prop}\label{Thun}
Let $F(x,y)\in\mathbb{Z}[x,y]$ be an irreducible  binary form with  $s+1$ nonzero coefficients and degree $n\geq 3$. 
Let $m$ be a positive integer.
 If $|D(F)|^{1/n(n-1)}\leq \log\log m$ and $\log\log m\geq (n(n-1))^8$, then $|F(x,y)|\leq m$ 
  has $ \ll m^{2/n}$ solutions.
\end{prop}

%-----------------------------------------------------------------------------------------------------------

\section{Proof of Proposition \ref{p2small}}\label{ssmall2}
%------------------------------------------------------------------------------------------------------

We apply \eqref{takinglog} to $Y=Y_0$ defined by \eqref{YS2}. We have,  by \eqref{m0} and \eqref{R},
\begin{align*}
n\log Y_0 + n\log (6R+5) &\ll  n\left(\log \frac{M}{6^nm} + \log 6^n \right) + n\log^3n \\
&\ll   n\log \frac{M}{6^n m} + n^2 
\end{align*}
and
$$
\log (M/(6^nm)) \gg n.
$$
Hence $|\mathfrak{X}\cup\mathfrak{X}_0| \ll  n$ and
we conclude Proposition \ref{p2small}.

\section{Large solutions, proof of Proposition \ref{p2large}}\label{slarge}

Let $F(x,y)=\sum_{i=0}^n a_ix^iy^{n-i}\in\Z[x,y]$ be an irreducible binary form that lies in $C(4s-2)$ and satisfies \eqref{D0}. 
Let $m$ be a positive integer that satisfies 
\eqref{m0}.

The following lemma is a version of the Lewis-Mahler inequality \cite{LM}, refined by
Bombieri and Schmidt \cite[Lemma 4]{Bom} and later written by Stewart \cite[Lemma 3]{Ste} in
terms of the discriminant of $F$ instead of the height. 

\begin{lemma}\label{Ste}
 For every pair of integers $(x, y)$ with $y\neq 0$,
$$
\min_{1\leq i\leq n} \left|\alpha_i-\dfrac{x}{y}\right| \leq \dfrac{2^{n-1} n^{(n-1)/2} M(F)^{n-2} |F(x,y)|}{|D(F)|^{1/2} y^n}.
$$
 \end{lemma}

 Let $(x,y)$ be a solution to \eqref{P} with $y>Y_0$, and let
 $$
 \left|\alpha_j-\dfrac{x}{y}\right| = \min_{1\leq i\leq t} \left|\alpha_i-\dfrac{x}{y}\right|,
 $$
 where $t$ is the cardinal of the set $\bm S$ defined by Lemma \ref{lmMS88}.
 By Lemmas \ref{Ste} and \ref{lmMS88},
$$
\left|\alpha_j-\dfrac{x}{y}\right| \leq   \dfrac{R 2^{n-1} n^{(n-1)/2} M(F)^{n-2} m}{|D(F)|^{1/2} y^n},
$$
with $R$ defined in \eqref{R}. On noticing that $e^{200n(n-1)}> R2^n n^{(n-1)/2}$ and 
using \eqref{m0}, we have 
 \begin{align}
  \left|\alpha_j-\dfrac{x}{y}\right|   &< \dfrac{(M(F)/m)^{n-2}}{2y^n}\label{gap0}\\
  &\leq \dfrac{1}{2y^{n-(n-2)/5}}\label{gap1},
 \end{align}
 where in the last inequality we used \eqref{YS2} and $y>Y_0$.
 Let $(x_1,y_1), (x_2,y_2),\ldots$ be the primitive solutions to \eqref{P} with $y>Y_0$ and ordered so that
 $$
 Y_0 < y_1\leq y_2\leq \ldots.
 $$
For all $i>1$, by \eqref{gap1} we have
 $$
 \dfrac{1}{y_iy_{i-1}} \leq \left|\dfrac{x_i}{y_i}-\dfrac{x_{i-1}}{y_{i-1}}\right| <\dfrac{1}{y_{i-1}^{n-(n-2)/5}}.
 $$
 Thus 
 \begin{equation}\label{3/2}
 y_{i}> y_{i-1}^{n-1-(n-2)/5}=y_{i-1}^{(4n-3)/5} > (M/m)^{4n-3}.
 \end{equation}
 On noting that $n-3\sqrt{n}/2>3/10$ for $n\geq 3$ and $\frac{3}{10}(4n-3)>n-2$,
 by \eqref{3/2} we have 
 $$y_i^{n-3\sqrt{n}/2}> y_i^{3/10} > (M(F)/m)^{n-2}.
 $$
 Hence, by \eqref{gap0},
 \begin{equation}\label{gap}
   |\alpha_j-x_i/y_i| <y_i^{-3\sqrt{n}/2}
   \end{equation}
   for all $i>1$. By \cite[Theorem 9A, Chapter 2]{Sbook}, the number of solutions to \eqref{gap} is 
   $$\ll 1+\log\log h(\alpha_j)/\log n,$$
   where $h(\alpha_j)$ is the absolute height of $\alpha_j$ defined in \cite[\S 7 Chapter 1]{Sbook}. By \cite[Lemma 2A Chapter 3]{Sbook},
   $$
   h(\alpha_j)=\mathrm{cont}(F)^{-1} (|a_n|\prod_{k=1}^n \sqrt{1+|\alpha_k|^2})^{1/n}.
   $$
   Here $\mathrm{cont}(F)=\mathrm{gcd}(|a_0|,\ldots,|a_n|)$.
   %and $H_K(\alpha_j)$ is the field height, both 
    If $y_i\geq h(\alpha_j)$, then the number of solutions is $\ll 1$. 
%   The absolute height differs from the Mahler measure of $F$ by a constant, since the only difference in their definitions
%   is that the first uses the Euclidean and the maximum norm for Archimedean and non-Archimedean absolute values respectively, 
%   whereas     Mahler measure involves   only the maximum norm for both Archimedean and non-Archimedean absolute values.
 
 Hence, if $(M(F)/m)^5\geq h(\alpha_j)$, the number of solutions is
$\ll 1$ and otherwise we have
$h(\alpha_j)\geq (M(F)/m)^5\gg (h(\alpha_j)^n/m)^5$, so $\log\log h(\alpha_j)\ll \log \log m^{1/n}$.

   Finally we conclude that the number of primitive large solutions is $\ll t(1+\log\log m^{1/n}) \ll s(1+\log\log m^{1/n})$.

%------------------------------------------------------------------------------------------------------------------------------------------------

\section{Bound for $\tilde P(F,m)$, proof of Theorem \ref{th1}}\label{sP1}

%------------------------------------------------------------------------------------------------------------------------------------------------ 

 \subsection{Definitions of small, medium and large solutions}

Let $\bm x = (x , y)\in\mathbb{Z}^2$. We define
$$
\xx = \max (|x|,|y|),\qquad
\x = \min(|x|,|y|).$$ 
 
Given $F\in C(4s-2)$,  we measure the size of possible solutions $(x , y)$ of \eqref{P} by the size of 
$\x$ and $\xx$.

 Relative to two quantities $Y_S$, $Y_L$, which will be defined below in \eqref{YS} and \eqref{YL}, we call a 
solution $(x , y) \in \mathbb{Z}^2$ 
\begin{eqnarray}\label{dsmall}
\textrm{small} \, \, &\textrm{if}& \, \, 0 \leq \x\leq Y_S,  \\
\textrm{medium} \, \, &\textrm{if}& \, \,  \xx\leq Y_L \ \textrm{and}\   \x> Y_S,\nonumber\\
\textrm{large} \, \, &\textrm{if}& \, \,  \xx >Y_L.\nonumber
\end{eqnarray}

We will split the count
 of possible solutions $(x , y)$ into \textit{small}, \textit{medium} and \textit{large} solutions. 
We choose the constants below to be consistent with Mueller and Schmidt's work \cite{MS88}. %These choices will allow us to compare $Y_S$ and $Y_L$ 
%with the 
%corresponding quantities  in \cite{MS88}, and to use some of the results that are already established there.
 Let $H(F)$ be the height of $F$ and let $M$ be the smallest Mahler measure among the forms equivalent to $F$.
 Put
\begin{equation}\label{defC}
C=Rm(2H(F)\sqrt{n(n+1)})^n,
\end{equation}
where $R$ is defined by \eqref{R}.
Pick numbers $a,b$ with $0<a<b<1$ so small that
\begin{equation}\label{ab}
\dfrac{\sqrt{2}\sqrt{3+a^2}}{1-b}<3.
\end{equation}
Put
$$%\begin{equation}\label{deflambda}
\lambda=\dfrac{\sqrt{2(n+a^2)}}{1-b},
%\end{equation}
$$
so that, by \eqref{ab},
$n-\lambda>0.$
Note that
\begin{equation}\label{lambda}
\lambda\asymp \sqrt{n}, \qquad n-\lambda\asymp n-\sqrt{n}.
\end{equation}
 We  define
\begin{equation}\label{defA}
A=\dfrac{1}{a^2}(\log M+\dfrac{n}{2}),
\end{equation}
and
\begin{equation}\label{YS}
Y_S=((e^6s)^nR^{2s}m)^{\frac{1}{n-2s}},
\end{equation}
%\begin{equation}\label{YL}
%Y_L=(2C)^{1/(n-\lambda)} (4e^A)^{\lambda/(n-\lambda)},
% \end{equation}
\begin{equation}\label{YL}
 Y_L=(2C)^{1/(n-\lambda)}(4e^A)^{\lambda/(n-\lambda)}.
 \end{equation}
% Note that
% \begin{equation}\label{Y3H}
%Y_L \ll H m^{1/n}.
%\end{equation}
 
Our definitions of  $Y_S$ and $Y_L$ are the same as the quantities introduced by Mueller and Schmidt  in \cite[eq. 2.10 and 2.9]{MS88}  to 
distinguish between small and large solutions.
 With these definitions we have:
 
  \begin{prop}[Mueller-Schmidt]\label{pplarge} Let $F(x,y)\in\mathbb{Z}[x,y]$ be an irreducible binary form with $s+1$ nonzero coefficients 
  and degree $n\geq 3$.
  For any positive integer $m$, 
the number of primitive large solutions of $F(x,y)\leq m$  is $\ll s$.
\end{prop}

 \begin{prop}\label{ppmedium} 
Let $F(x,y)\in\mathbb{Z}[x,y]$ be an irreducible binary form with $s+1$ nonzero coefficients and degree $n\geq 3s$. 
  For any positive integer $m$, 
the number of primitive medium solutions of $F(x,y)\leq m$  is 
$$
\ll \left\{
\begin{array}{ll}
 s(1+\frac{\log m^{1/n}}{\log H(F)}) &\quad\mbox{if $n \geq s^4$}\\
 (s\log s)(1+\frac{\log m^{1/n}}{\log H(F)}) &\quad\mbox{if $9s^2 \leq n<s^4$}\\
 (s\log s)(1+\frac{s+\log m^{1/n}}{\log H(F)}) &\quad\mbox{if $ n<9s^2$}.
\end{array}
\right.
$$
\end{prop}

%\begin{prop}\label{ppsmall}
% The number of primitive small solutions of \eqref{Tin} is $\ll (s\log s)m^{2/n} \log m$.
%\end{prop}

Proposition \ref{pplarge} is Mueller-Schmidt's result \cite[Prop. 1]{MS88}. Proposition \ref{ppmedium} will be proved in section \ref{smed} 
of this article. 
We generalise to  arbitrary $m$ the argument in \cite[Section 5]{AB}. 
%We prove Propositions \ref{ppsmall} and \ref{ppmedium} in Sections \ref{IIsmall1}-\ref{IIsmall2} and \ref{IImed} respectively. 
 Small solutions are more difficult to deal with; they will be counted essentially by the proposition \ref{ppsmall} below that we will prove in sections 
 \ref{ssmall} and \ref{ssmall1}.
\\

We write $\tilde P_{S_x}(F,m)$, $\tilde P_{S_y}(F,m)$ for the number of primitive  solutions of \eqref{P} with $0\leq x\leq Y_S$ and 
$0\leq y\leq Y_S$ respectively.

\begin{prop}\label{ppsmall} 
Let $F(x,y)\in\mathbb{Z}[x,y]$ be an irreducible binary form that lies in $C(4s-2)$. Let $M$ be the smallest Mahler measure
among the forms equivalent to $F$, and  $m$ be a positive integer such that
\begin{equation}\label{m2}
 m\leq \dfrac{M}{100^n}.
\end{equation} 
Then 
\begin{equation}\label{PxPy}
 \max(\tilde P_{S_x}(F,m),\tilde P_{S_y}(F,m)) \ll s+ \log^3 n + \log m^{1/n}.
 \end{equation}
\end{prop}

Next we prove Theorem \ref{th1} assuming Propositions \ref{ppmedium} and \ref{ppsmall}.
Large and medium solutions are counted by Propositions \ref{pplarge} and \ref{ppmedium} respectively, so we only need to count small solutions assuming 
Proposition \ref{ppsmall}. For this we  refine the  argument that we  used for large discriminant in section \ref{sP2}. 
Note that the argument below can be used together with Proposition \ref{ppsmall} to bound small solutions because the condition on $F$ 
in Proposition \ref{ppsmall} is
 $F\in C(4s-2)$; it would not be useful if we had a condition on the number of nonzero coefficients such as in Proposition \ref{ppmedium}.

Let $F\in\mathbb{Z}[x,y]$ be an irreducible binary form with $s+1$ nonzero coefficients and degree $n\geq 3s$. By Lemma \ref{Ct}, $F(x,y)\in C(4s-2)$.
Let $m$ be a positive integer.
Pick  the smallest prime $p$ that satisfies 
\begin{equation}\label{yup}
p>10^6 m^\frac{2}{n}|D(F)|^{-\frac{1}{n(n-1)}}
\end{equation}
and consider the matrices
$$
A_0=\begin{pmatrix} 1 &0\\0 &p\end{pmatrix},\qquad
A_j=\begin{pmatrix} p &j\\0 &1\end{pmatrix}\qquad (j=1,\ldots,p).
$$
Recall that
$F_{A_j}\in C(4s-2)$ for $j=0,\ldots,p$ and  $\mathbb{Z}^2=\cup_{j=0}^p A_j\mathbb{Z}^2$, so that
any solution 
$(x,y)$ of \eqref{P} gives a solution $(u,v)=A_j^{-1}(x,y)$ of $2^{-n}m\leq F_{A_j}(u,v)< m$ for some $j=0,\ldots, p$. 
Moreover, when $j\geq 1$ we have $v=y$,
so if $(x,y)$ satisfies $0\leq y\leq Y_S$, then also $(u,v)$ satisfies $0\leq v\leq Y_S$.
Hence
$$
\tilde P_{S_y}(F,m)\leq \tilde P(F_{A_0},m) + \sum_{j=1}^p \tilde P_{S_y}(F_{A_j},m).
$$

For $j=0,\ldots,p$, by \eqref{St6} and \eqref{yup}, 
$$
|D(F_{A_j})|=p^{n(n-1)}|D(F)|> 10^{6n(n-1)} m^{2n-2}.
$$
Then, by \eqref{mahD5},
\begin{equation}\label{Dp}
10^{6n(n-1)} m^{2n-2}\leq |D(F_{A_j})| \leq n^n M_j^{2n-2},
\end{equation}
where $M_j$ is the smallest Mahler measure among the forms equivalent to $F_{A_j}$.
Hence, $m\leq M_j/100^n$, and we can apply Proposition \ref{ppsmall} to $F_{A_j}$. We obtain that 
$$
\tilde P_{S_y}(F_{A_j},m)\ll s + \log^3n + \log m^{1/n}.
$$
Therefore,
\begin{align*}
\tilde P_{S_y}(F,m) &\ll \tilde P(F_{A_0},m) + p(s + \log^3n + \log m^{1/n})\\
&\ll \tilde P(F_{A_0},m)+ m^{2/n}|D(F)|^{-\frac{1}{n(n-1)}}(s+\log^3n + \log m^{1/n}).
\end{align*}
If we consider 
$$
A'_0=\begin{pmatrix} p &0\\0 &1\end{pmatrix},\qquad
A'_j=\begin{pmatrix} 1 &0\\j &p\end{pmatrix}\qquad (j=1,\ldots,p),
$$
any solution of \eqref{P} gives a solution $(u,v)$ of $2^{-n} m\leq F_{A_j'}(u,v)< m$ for some $j=0,\ldots,p$. Similarly as above, 
 when $j\geq 1$, $u=x$ and we conclude that
$$
\tilde P_{S_x}(F,m)\ll  \tilde P(F'_{A_0},m)+ m^{2/n}|D(F)|^{-\frac{1}{n(n-1)}}(s+\log^3n + \log m^{1/n}).
$$
The forms $F_{A_0}$ and $F'_{A_0}$ have $s+1$ nonzero coefficients, 
so  Propositions \ref{ppmedium} and \ref{pplarge} apply together with Proposition \ref{ppsmall}. We obtain that  
$$
\max(\tilde P(F'_{A_0},m), \tilde P(F_{A_0},m))\ll c(s)(1+\log m^{1/n}) + \log^3 n,
$$
where 
$c(s)$ is defined in \eqref{c}.

Finally, the number of primitive small solutions of \eqref{P} is 
%\begin{equation}\label{boundP}
$$
\ll  (c(s) (1+\log m^{1/n})+\log^3 n) m^{2/n}|D(F)|^{-\frac{1}{n(n-1)}}.
$$
%\end{equation}

%-----------------------------------------------------------------------------------------------------
\section{Proof of Proposition \ref{ppsmall}}\label{ssmall1}
%--------------------------------------------------------------------------------------------------------

We apply the inequality \eqref{takinglog} to $Y=Y_S$, defined by \eqref{YS}.
Since $\frac{n}{n-2s}\leq 3$,
$$
n\log Y_S + n\log(6R+5)\ll \log m +n\log s + n\log^3 n.
$$
We also have, by \eqref{m2}, 
$$
\log (M/(6^nm)) \gg n.
$$
We conclude that
$$
|\mathfrak{X}\cup\mathfrak{X}_0|\ll \dfrac{\log m}{n} +\log s + \log^3 n.
$$
Hence 
$$
\max(\tilde P_{S_x}(F,m), \tilde P_{S_y}(F,m)) \ll  s + \log^3 n + \log m^{1/n}.
$$

%--------------------------------------------------------------------------------
\section{Medium solutions, proof of Proposition \ref{ppmedium}}\label{smed}
%----------------------------------------------------------------------------------

Let $F(x,y)\in\Z[x,y]$ be a binary form of degree $n$ with $s$ non-zero coefficients. Let $m$ be a positive integer.

We divide the interval $ [Y_{S}, Y_{L}]$ into $N+1$ subintervals, where $Y_{S}$ and $Y_{L}$ are defined in \eqref{YS} and \eqref{YL} and 
$N$ depends on $s$ and is defined below. We will show that there are only few solutions $(x , y)$ with $\x$   in each of these subintervals. 
In this section we will assume $n \geq 3s$.
We define  the positive integer $N = N(n, s)$ as follows.\\
If $n \geq s^4$, we put $N=2$. Otherwise, we put $k=\sqrt{n}$ if $9s^2\leq n < s^4$ and $k=n$ if $n<9s^2$, and
 choose $N \in \mathbb{N}$ such that
\begin{equation}\label{N}
 3s^{1+\frac{1}{N}}\leq k\leq 3s^{1+\frac{1}{N-1}}.
\end{equation}
The inequality \eqref{N} leads to
\begin{equation}\label{bN}
N\leq \dfrac{\log s}{\log k - \log s}.
\end{equation}

For $\ell=1,\ldots,N$, we define 
$$
Y_\ell=Y_S H(F)^{\frac{1}{s^{1-(\ell-1)/N}}}.
$$
We put 
$$Y_0=Y_S\, \qquad  \textrm{and} \, \qquad Y_{N+1}=Y_L.
$$

\begin{prop}\label{thesetT}
 There is a set $\bm T$ of roots of $F(x,1)$ and a set $\bm T^\ast$ of roots of $F(1,y)$, both with cardinalities $\leq 6s+4$, 
such that any solution $(x,y)$ of $|F(x,y)|\leq m$ with $\x\geq Y_S$ either has 
\begin{equation}\label{17.1}
\left|\alpha-\dfrac{x}{y}\right|<\dfrac{R(ns)^2}{H(F)^{(1/s)-(1/n)}}\left(\dfrac{(4e^3s)^n m}{y^{n}}\right)^{1/s}
\end{equation}
with some $\alpha\in \bm T$ and $R$ defined by \eqref{R}, or has
\begin{equation}\label{17.2}
\left|\alpha^\ast-\dfrac{y}{x}\right|<\dfrac{R(ns)^2}{H(F)^{(1/s)-(1/n)}}\left(\dfrac{(4e^3s)^n m}{x^{n}}\right)^{1/s}
\end{equation}
for some $\alpha^\ast\in \bm T^\ast$.
\end{prop}

\begin{proof}
This is Lemma 17 of \cite{MS88}.  
\end{proof}

Let $\alpha\in \bm T$. For $\ell\in\left\{0,\ldots,N\right\}$, let
$(x_1, y_1),\ldots, (x_{w_{\ell}}, y_{w_{\ell}})$ be the  primitive solutions of $|F(x,y)|\leq m$, 
with  $Y_{\ell} < y_i\leq Y_{\ell+1}$, satisfying  \eqref{17.1} and ordered so that
$$
 Y_{\ell}< y_{1}\leq
\ldots\leq y_{w_{\ell}} \leq Y_{\ell+1}.
$$
By \eqref{17.1},
we have that
$$
\dfrac{1}{y_i y_{i+1}} \leq \left|\dfrac{x_{i+1}}{y_{i+1}}-\dfrac{x_i}{y_i}\right|\leq \dfrac{U}{ H(F)^{\frac{1}{s}-\frac{1}{n}} y_i^{\frac{n}{s}}},
$$
with 
$$
U=2R(ns)^2 (4e^3s)^{n/s} m^{1/s}.
$$
Therefore, for solutions $(x, y)$ with $y \in (Y_{\ell}, Y_{\ell+1}]$, we have
\begin{equation}\label{yi1}
y_{i+1}\geq U^ {-1} H(F)^{\frac{1}{s}-\frac{1}{n}} y_i^{\frac{n}{s}-1}\geq U^{-1} H(F)^{\frac{1}{s}-\frac{1}{n}} Y_\ell^{\frac{n}{s}-2}y_i.
\end{equation}

First we will give an estimate for the number of primitive solutions in $(Y_{0}, Y_{1}]$. By the definition 
 of $Y_S=Y_0$ and since $n\geq 3s$, we have
\begin{equation}\label{yi2}
U^{-1}Y_S^{\frac{n}{s}-2} \geq \dfrac{e^{\frac{3n}{s}}R}{2(ns)^2 4^{\frac{n}{s}}} \geq 1.
\end{equation}

For $\ell=0$, we have by \eqref{yi1} and \eqref{yi2} that $y_{i+1}\geq H(F)^{\frac{1}{s}-\frac{1}{n}} y_i$, so 
$y_{w_0}\geq (H(F)^{\frac{1}{s}-\frac{1}{n}})^{(w_0-1)} y_1$.  So we have
$$
Y_{1} \geq y_{w_0}\geq (H(F)^{\frac{1}{s}-\frac{1}{n}})^{(w_0-1)} Y_{0},
$$
and
$$
w_0-1\leq \dfrac{\log \frac{Y_1}{Y_0}}{(\frac{1}{s}-\frac{1}{n})\log H(F)} < \dfrac{1}{1-\frac{s}{n}}\leq \dfrac{3}{2},
$$
since 
$\log \frac{Y_1}{Y_0} = \frac{1}{s}\log H(F)$ and $n\geq 3s$.

For $1\leq \ell < N$, by \eqref{yi1} and \eqref{yi2} we have that
$$
 y_{i+1}\geq U^{-1} H(F)^{\frac{1}{s}-\frac{1}{n}} Y_S^{\frac{n}{s}-2} H(F)^{\frac{n/s-2}{s^{1-(\ell-1)/N}}} y_i
 \geq H(F)^{\frac{n}{s^{2-(\ell-1)/N}}-\frac{2}{s^{1-(\ell-1)/N}}+\frac{1}{s}-\frac{1}{n}} y_i.
 $$
Therefore,
$$
y_{w_{\ell}}\geq  H(F)^{\left(\frac{n}{s^{2-(\ell-1)/N}}-\frac{2}{s^{1-(\ell-1)/N}}+\frac{1}{s}-\frac{1}{n}\right)(w_{\ell}-1)}y_1,
$$
and since $Y_{\ell} < y_{1} \leq y_{w_{\ell}} \leq Y_{\ell+1}$, we have
$$
w_{\ell}-1\leq \dfrac{\log \frac{Y_{\ell+1}}{Y_\ell}}{(\frac{n}{s^{2-(\ell-1)/N}}-\frac{2}{s^{1-(\ell-1)/N}}+\frac{1}{s}-\frac{1}{n})\log H(F)}.
$$
For $\ell < N$, since $\log \frac{Y_{\ell+1}}{Y_\ell} < \frac{1}{s^{1-\ell/N}}\log H(F)$ and $n\geq 3s^{1+1/N}$,
$$
w_{\ell}-1\leq
\dfrac{1}{\frac{n}{s^{1+1/N}}-\frac{2}{s^{1/N}}+\frac{1}{s^{\ell/N}}-\frac{s^{1-\ell/N}}{n}} \leq 1.
$$
For $\ell=N$, we have  
$$
\log Y_{\ell+1}=\log Y_L  \ll \log H(F) +  \log m^{1/n} + \sqrt{n},
$$ 
so
$$
w_{N}-1\ll \dfrac{\log H(F) +  \log m^{1/n} + \sqrt{n}}{(\frac{n}{s^{1+1/N}}-\frac{2}{s^{1/N}}+\frac{1}{s}-\frac{1}{n})\log H(F)}.
$$
If $n<9s^2$, then $\sqrt{n}<3s$ and 
$$
w_{N}-1\ll 1+\dfrac{s+\log m^{1/n}}{\log H(F)}.
$$
If $n\geq 9s^2$, then  by \eqref{N},
$$
\dfrac{\sqrt{n}}{\frac{n}{s^{1+1/N}}-\frac{2}{s^{1/N}}+\frac{1}{s}-\frac{1}{n}}\ll 1,
$$ and
$$
w_N-1 \ll 1+\dfrac{\log m^{1/n}}{\log H(F)}.
$$

We conclude that the number of primitive medium solutions of \eqref{17.1}  for each $\alpha\in \bm T$ is 
$\ll N+ \frac{\log m^{1/n}}{\log H(F)}$ when
$n\geq 9 s^2$ and $\ll N+\frac{s+\log m^{1/n}}{\log H(F)}$ when
$n< 9 s^2$. 
In a similar way, we obtain the same bound for the number of primitive medium solutions of \eqref{17.2}  for each $\alpha^\ast\in \bm T^\ast$. 
%By Proposition \ref{thesetT}, the number of primitive medium solutions is 
%$$
%\ll N\left(1+\dfrac{\log m^{1/n}}{\log H}\right).
%$$
Using Proposition \ref{thesetT} and the fact that $N=2$ for $n\geq s^4$ and $N\leq \log s$ for $n<s^4$, we obtain Proposition  \ref{ppmedium}.

%------------------------------------------------------------------------------

   %-----------------------------------------

\end{document}